\documentclass[12pt]{amsart}
\usepackage{amsmath,amssymb,amsbsy,amsfonts,latexsym,amsopn,amstext,cite,
                                               amsxtra,euscript,amscd,bm}
\usepackage{url}

\usepackage{mathrsfs}

\usepackage{color}
\usepackage[colorlinks,linkcolor=blue,anchorcolor=blue,citecolor=blue,backref=page]{hyperref}
\usepackage{color}
\usepackage{graphics,epsfig}
\usepackage{graphicx}
\usepackage{float}
\usepackage{epstopdf}
\hypersetup{breaklinks=true}

\usepackage[np]{numprint}
\npdecimalsign{\ensuremath{.}}

\usepackage{bibentry}

\usepackage[english]{babel}
\usepackage{mathtools}
\usepackage{todonotes}
\usepackage{url}
\usepackage[colorlinks,linkcolor=blue,anchorcolor=blue,citecolor=blue,backref=page]{hyperref}

\usepackage[norefs,nocites]{refcheck}

\begin{document}

\newtheorem{theorem}{Theorem}
\newtheorem{lemma}[theorem]{Lemma}
\newtheorem{claim}[theorem]{Claim}
\newtheorem{cor}[theorem]{Corollary}
\newtheorem{conj}[theorem]{Conjecture}
\newtheorem{prop}[theorem]{Proposition}
\newtheorem{definition}[theorem]{Definition}
\newtheorem{question}[theorem]{Question}
\newtheorem{example}[theorem]{Example}
\newcommand{\hh}{{{\mathrm h}}}
\newtheorem{remark}[theorem]{Remark}

\numberwithin{equation}{section}
\numberwithin{theorem}{section}
\numberwithin{table}{section}
\numberwithin{figure}{section}

\def\sssum{\mathop{\sum\!\sum\!\sum}}
\def\ssum{\mathop{\sum\ldots \sum}}
\def\iint{\mathop{\int\ldots \int}}

\newcommand{\diam}{\operatorname{diam}}

\def\squareforqed{\hbox{\rlap{$\sqcap$}$\sqcup$}}
\def\qed{\ifmmode\squareforqed\else{\unskip\nobreak\hfil
\penalty50\hskip1em \nobreak\hfil\squareforqed
\parfillskip=0pt\finalhyphendemerits=0\endgraf}\fi}

\newfont{\teneufm}{eufm10}
\newfont{\seveneufm}{eufm7}
\newfont{\fiveeufm}{eufm5}
%
%
\newfam\eufmfam
     \textfont\eufmfam=\teneufm
\scriptfont\eufmfam=\seveneufm
     \scriptscriptfont\eufmfam=\fiveeufm
%
%
\def\frak#1{{\fam\eufmfam\relax#1}}

\newcommand{\bflambda}{{\boldsymbol{\lambda}}}
\newcommand{\bfmu}{{\boldsymbol{\mu}}}
\newcommand{\bfxi}{{\boldsymbol{\eta}}}
\newcommand{\bfrho}{{\boldsymbol{\rho}}}

\def\eps{\varepsilon}

\def\fK{\mathfrak K}
\def\fT{\mathfrak{T}}
\def\fL{\mathfrak L}
\def\fR{\mathfrak R}

\def\fA{{\mathfrak A}}
\def\fB{{\mathfrak B}}
\def\fC{{\mathfrak C}}
\def\fM{{\mathfrak M}}
\def\fS{{\mathfrak  S}}
\def\fU{{\mathfrak U}}

\def\T {\mathsf {T}}
\def\Tor{\mathsf{T}_d}
\def\Tore{\widetilde{\mathrm{T}}_{d} }

\def\sM {\mathsf {M}}

\def\ss{\mathsf {s}}

\def\Kmnd{\cK_d(m,n)}
\def\Kmnp{\cK_p(m,n)}
\def\Kmnq{\cK_q(m,n)}

\def \balpha{\bm{\alpha}}
\def \bbeta{\bm{\beta}}
\def \bgamma{\bm{\gamma}}
\def \bdelta{\bm{\delta}}
\def \bzeta{\bm{\zeta}}
\def \blambda{\bm{\lambda}}
\def \bchi{\bm{\chi}}
\def \bphi{\bm{\varphi}}
\def \bpsi{\bm{\psi}}
\def \bnu{\bm{\nu}}
\def \bomega{\bm{\omega}}

\def \bell{\bm{\ell}}

\def\eqref#1{(\ref{#1})}

\def\vec#1{\mathbf{#1}}

\newcommand{\abs}[1]{\left| #1 \right|}

\def\Zq{\mathbb{Z}_q}
\def\Zqx{\mathbb{Z}_q^*}
\def\Zd{\mathbb{Z}_d}
\def\Zdx{\mathbb{Z}_d^*}
\def\Zf{\mathbb{Z}_f}
\def\Zfx{\mathbb{Z}_f^*}
\def\Zp{\mathbb{Z}_p}
\def\Zpx{\mathbb{Z}_p^*}
\def\cM{\mathcal M}
\def\cE{\mathcal E}
\def\cH{\mathcal H}

\def\le{\leqslant}

\def\ge{\geqslant}

\def\sfB{\mathsf {B}}
\def\sfC{\mathsf {C}}
\def\L{\mathsf {L}}
\def\FF{\mathsf {F}}

\def\sE {\mathscr{E}}
\def\sS {\mathscr{S}}

\def\cA{{\mathcal A}}
\def\cB{{\mathcal B}}
\def\cC{{\mathcal C}}
\def\cD{{\mathcal D}}
\def\cE{{\mathcal E}}
\def\cF{{\mathcal F}}
\def\cG{{\mathcal G}}
\def\cH{{\mathcal H}}
\def\cI{{\mathcal I}}
\def\cJ{{\mathcal J}}
\def\cK{{\mathcal K}}
\def\cL{{\mathcal L}}
\def\cM{{\mathcal M}}
\def\cN{{\mathcal N}}
\def\cO{{\mathcal O}}
\def\cP{{\mathcal P}}
\def\cQ{{\mathcal Q}}
\def\cR{{\mathcal R}}
\def\cS{{\mathcal S}}
\def\cT{{\mathcal T}}
\def\cU{{\mathcal U}}
\def\cV{{\mathcal V}}
\def\cW{{\mathcal W}}
\def\cX{{\mathcal X}}
\def\cY{{\mathcal Y}}
\def\cZ{{\mathcal Z}}
\newcommand{\rmod}[1]{\: \mbox{mod} \: #1}

\def\cg{{\mathcal g}}

\def\vy{\mathbf y}
\def\vr{\mathbf r}
\def\vx{\mathbf x}
\def\va{\mathbf a}
\def\vb{\mathbf b}
\def\vc{\mathbf c}
\def\ve{\mathbf e}
\def\vh{\mathbf h}
\def\vk{\mathbf k}
\def\vm{\mathbf m}
\def\vz{\mathbf z}
\def\vu{\mathbf u}
\def\vv{\mathbf v}

\def\e{{\mathbf{\,e}}}
\def\ep{{\mathbf{\,e}}_p}
\def\eq{{\mathbf{\,e}}_q}

\def\Tr{{\mathrm{Tr}}}
\def\Nm{{\mathrm{Nm}}}

 \def\SS{{\mathbf{S}}}

\def\lcm{{\mathrm{lcm}}}

 \def\0{{\mathbf{0}}}

\def\({\left(}
\def\){\right)}
\def\l|{\left|}
\def\r|{\right|}
\def\fl#1{\left\lfloor#1\right\rfloor}
\def\rf#1{\left\lceil#1\right\rceil}
\def\sumstar#1{\mathop{\sum\vphantom|^{\!\!*}\,}_{#1}}

\def\mand{\qquad \mbox{and} \qquad}

\def\tblue#1{\begin{color}{blue}{{#1}}\end{color}}




\hyphenation{re-pub-lished}

\mathsurround=1pt

\def\bfdefault{b}

\def \F{{\mathbb F}}
\def \K{{\mathbb K}}
\def \N{{\mathbb N}}
\def \Z{{\mathbb Z}}
\def \P{{\mathbb P}}
\def \Q{{\mathbb Q}}
\def \R{{\mathbb R}}
\def \C{{\mathbb C}}
\def\Fp{\F_p}
\def \fp{\Fp^*}

 \def \xbar{\overline x}

\title[Hybrid mean value theorems]{On a Hybrid Version of the Vinogradov Mean Value Theorem}

 \author[C. Chen] {Changhao Chen}

\address{Department of Pure Mathematics, University of New South Wales,
Sydney, NSW 2052, Australia}
\email{changhao.chenm@gmail.com}

 \author[I. E. Shparlinski] {Igor E. Shparlinski}

\address{Department of Pure Mathematics, University of New South Wales,
Sydney, NSW 2052, Australia}
\email{igor.shparlinski@unsw.edu.au}

\begin{abstract}   Given a family $\bphi = \(\varphi_1, \ldots, \varphi_d\)\in \Z[T]^d$  of $d$ distinct 
nonconstant polynomials, a positive integer $k\le d$ and a real positive parameter $\rho$,  we consider the 
mean value 
$$
M_{k, \rho} (\bphi, N) = \int_{\vx \in [0,1]^k} \sup_{\vy \in [0,1]^{d-k}}
\left| S_{\bphi}( \vx, \vy; N)  \right|^\rho d\vx
$$
of exponential sums
$$
S_{\bphi}( \vx, \vy; N) = \sum_{n=1}^{N} \exp\(2 \pi i\(\sum_{j=1}^k x_j \varphi_j(n)+ \sum_{j=1}^{d-k}y_j\varphi_{k+j}(n)\)\), 
$$
where $\vx = (x_1, \ldots, x_k)$ and $\vy =(y_1, \ldots, y_{d-k})$.
The case of polynomials $\varphi_i(T) = T^i$, $i =1, \ldots, d$ and $k=d$ 
corresponds to the classical {\it Vinaogradov mean value theorem\/}. 

Here motivated by recent works of Wooley (2015) and the authors (2019) on bounds on $\sup_{\vy \in [0,1]^{d-k}}
\left| S_{\bphi}( \vx, \vy; N)  \right|$ for almost all $\vx \in [0,1]^k$, we obtain nontrivial bounds on $M_{k, \rho} (\bphi, N)$. 
\end{abstract}


\maketitle

%

\section{Introduction}

For an integer $\nu \geqslant 2$, let $\T_\nu = (\R/\Z)^\nu$ be the  $\nu$-dimensional unit torus. 
We denote 
$$
\e(x) = \exp(2\pi ix).
$$
The exponential 
sums  
$$
S_d(\vu; N)=\sum_{n=1}^N \e(u_1n+\ldots+u_d n^d), \quad \vu = (u_1, \ldots, u_d) \in \Tor, 
$$
introduced and estimated by Weyl~\cite{Weyl}, are  commonly called the~\emph{Weyl sums}.

Weyl sums  appear in a great variety of number theoretic problems starting with the problem of  uniformly of distribution 
of fractional parts of real polynomials, see also~\cite{Baker}.   They also play a crucial role in  
estimating the {\it zero-free region} of the {\it Riemann zeta-function} and thus in turn  in bounding in  the error term  in  the {\it prime number theorem}, see~\cite[Section~8.5]{IwKow},  
and  the {\it Waring problem}, see~\cite[Section~20.2]{IwKow}, in  estimating 
 short character sums modulo highly composite numbers~\cite[Section~12.6]{IwKow}.

Thanks to recent striking results   of  Bourgain, Demeter and Guth~\cite{BDG} (for $d \geqslant 4$) 
and Wooley~\cite{Wool2} (for $d=3$)  (see also~\cite{Wool5}), for  the mean value  of  $S_d(\vu; N)$ we have
\begin{equation}
\label{eq:MVT}
\int_{\Tor} |S_d(\vu; N)|^{2s(d)}d\vu \leqslant  N^{s(d)+o(1)}, \qquad N \to  \infty,
\end{equation}
where   
\begin{equation}
\label{eq:sq}
s(d)=\frac{d(d+1)}{2}, 
\end{equation}
which is the best possible  form  
of  the Vinogradov mean value theorem.

On the other hand,   for individual  sums it is known that their size depends on Diophantine properties 
of the coefficients $u_1, \ldots, u_d$, but generally the situation is not well understood, see~\cite{Brud,BD}.  
 
The following best known bound is a direct implication of~\eqref{eq:MVT} and is given in~\cite[Theorem~5]{Bourg}.   Let $\vu = (u_1, \ldots, u_d)  \in \Tor$ be such that for some $\nu$ with $2 \le \nu\le d $ and some positive integers $a$ and $q$ with $\gcd(a,q)=1$  we have
$$
\left| u_\nu - \frac{a}{q}\right| \le \frac{1}{q^2}. 
$$
Then for any $\varepsilon>0$ there exits a constant $C(\varepsilon)$ such that
$$
|S_d(\vu; N)| \le C(\varepsilon)N^{1+\varepsilon} \(q^{-1} + N^{-1} + qN^{-\nu}\)^{\frac{1}{d(d-1)}}.
$$

Recently, Wooley~\cite{Wool3} has considered a hybrid scenario which interpolates between {\it individual\/}  bounds 
and {\it mean value\/} estimates. In this settings one seeks results  
 which hold for {\it all\/} values of the components of  $\vu = (u_1, \ldots, u_d) \in \Tor$     
on some prescribed set of positions and  {\it almost all\/} values of the components on the remaining positions.

Given a family $\bphi = \(\varphi_1, \ldots, \varphi_d\)\in \Z[T]^d$  of $d$ distinct 
nonconstant polynomials and a sequence of complex  weights $\vec{a} = (a_n)_{n=1}^\infty$, for $\vu=(u_1, \ldots, u_d)\in \Tor$ we  define the 
trigonometric polynomials 
\begin{equation}
\label{eq:Tphi}
T_{\va, \bphi}( \vu; N)=\sum_{n=1}^{N}a_n \e\(u_1 \varphi_1(n)+\ldots + u_d\varphi_d(n) \).
\end{equation} 
Furthermore, for $k=1, \ldots, d$, we decompose
$$
\Tor = \T_k\times \T_{d-k}.
$$ 
Given  $\vx\in \T_k$, $\vy\in\T_{d-k}$ we refine the notation~\eqref{eq:Tphi}
and  write
\begin{equation}
\label{eq:split}
T_{\va, \bphi}( \vx, \vy; N)= \sum_{n=1}^{N}a_n \e\(\sum_{j=1}^k x_j \varphi_j(n)+ \sum_{j=1}^{d-k}y_j\varphi_{k+j}(n) \).
\end{equation}
If $\va = \ve =  (1)_{n=1}^\infty$ (that is,  $a_n =1$ for each $n\in \N$) we just write 
$$
T_{\bphi}(\vx, \vy; N) = T_{\ve, \bphi}( \vx, \vy; N).
$$

In fact  Wooley~\cite{Wool3} has studied only the classical case $a_n = 1$ for all $n\in \N$ and the polynomials
\begin{equation}
\label{eq:classical}
  \{\varphi_1(T),\ldots, \varphi_d(T)\}=\{T, \ldots, T^d\} .
\end{equation} 
(we note that the order of $\varphi_1,\ldots, \varphi_d$ is not specified in~\eqref{eq:classical}). 
Embedding  in the  argument  of Wooley~\cite[Theorem~1.1]{Wool3}  
the modern form of  the 
Vinogradov mean value theorem~\eqref{eq:MVT}, one derives   that for almost all $\vx\in \T_k $ 
with respect to the $k$-dimensional Lebesgue measure on $\T_k$, we have
\begin{equation}
\label{eq:wooley}
\sup_{\vy\in \T_{d-k}} |T_{\bphi}(\vx, \vy; N)| \le N^{(1+ \Delta_W(\bphi,k))/2 +o(1)}, \qquad N\to \infty, 
\end{equation}
where 
$$
\Delta_W(\bphi,k) =   \frac{2\sigma_k(\bphi)+d-k +1}{2s(d) +d-k+1}
$$
with  $s(d)$ given by~\eqref{eq:sq} and 
\begin{equation}
\label{eq:sigmak}
\sigma_k(\bphi)=\sum_{j=k+1}^d \deg \varphi_j . 
\end{equation}

The authors~\cite{ChSh-New} have extended and improved this and some other  results of 
Wooley~\cite{Wool3}. In particular, by~\cite[Theorems~2.1, 2.3 and Corollary~2.4]{ChSh-New},  
the bound~\eqref{eq:wooley} holds with 
$$
\Delta_{CS}(\bphi,k) = \min\left\{ \frac{2\sigma_k(\bphi)+d-k } {2s(d)  +d-k }, \frac{\sigma_k(\bphi)+1}{s(d)} \right\} < \Delta_W(\bphi,k)
$$ 
instead of $\Delta_W(\bphi,k)$ and also applies to   sums $T_{\va, \bphi}( \vx, \vy; N)$ 
with complex weights $a_n=n^{o(1)}, n\in \N$, and a large family  of polynomials.

\section{Hybrid mean value theorems}

\subsection{Notation and conventions}

Throughout the paper, the notation $U = O(V)$, 
$U \ll V$ and $ V\gg U$  are equivalent to $|U|\leqslant c V$ for some positive constant $c$, 
which throughout the paper may depend on the degree $d$ and occasionally on the small real positive 
parameter $\varepsilon$. 

For any quantity $V> 1$ we write $U = V^{o(1)}$ (as $V \to \infty$) to indicate a function of $V$ which 
satisfies $|U| \le V^\eps$ for any $\eps> 0$, provided $V$ is large enough. One additional advantage 
of using $V^{o(1)}$ is that it absorbs $\log V$ and other similar quantities without changing  the whole 
expression.

 We use $\# \cS$ to denote the cardinality of a finite set $\cS$.

We say that some property holds for almost all $\vx \in \T_k$ if it holds for a set 
 $\cX \subseteq \T_k$ of $k$-dimensional  Lebesgue measure  $\lambda(\cX) = 1$.

\subsection{Main results}

Here, instead of asking about bounds on
 $$
 \sup_{\vy \in \T_{d-k}} \left| T_{\va, \bphi}( \vx, \vy; N) \right|,
 $$
  for a real $\rho> 0 $ we consider 
the mean value
$$
M_{k,\rho} (\va, \bphi, N) = \int_{\vx \in \T_k} \sup_{\vy \in \T_{d-k}}
 \left| T_{\va, \bphi}( \vx, \vy; N) \right|^\rho d\vx. 
$$

Let 
$$
W(T;\bphi) = \det\(\varphi_i^{(j-1)}(T)\)_{i,j=1}^n
$$
denote   the Wronskian  of  $d$ polynomials $\bphi = \(\varphi_1, \ldots, \varphi_d\)\in \Z[T]^d$.

We recall that $s(d)$ and 
$\sigma_k(\bphi)$ are given by~\eqref{eq:sq} and~\eqref{eq:sigmak}, respectively.

\begin{theorem}
\label{thm:general} 
Suppose that $\bphi \in \Z[T]^d$ is such that 
  the Wronskian  $W(T;\bphi)$ 
does not vanish identically. Let   $\vec{a} = (a_n)_{n=1}^\infty$ 
be a sequence of complex  weights  with $a_n=n^{o(1)}$. 
Then for any real positive $\rho \le  2 s(d) +d -k$ we have 
$$
M_{k,\rho} (\va, \bphi, N) \le N^{  \mu(\bphi,k) \rho  +o(1)},  \qquad N\to \infty, 
$$ 
where  
$$
\mu (\bphi,k) = \frac{s(d)+\sigma_k(\bphi)+d-k } {2 s(d) +d -k}.
$$ 
\end{theorem}

Note that Theorem~\ref{thm:general} gives a non-trivial bound, that is, $\mu(\bphi, k)<1$, provided that
 $\sigma_k(\bphi)<s(d)$. Moreover for $k=d$, we have $\sigma_d(\bphi)=0$ and thus we recover   the bound~\eqref{eq:MVT}
 in the Vinogradov mean value theorem.

From Theorem~\ref{thm:general} we derive the following  bounds for the mean values of  short sums. For $K\in \Z$, we consider Weyl sums  over short intervals 
$$
S_d(\vu; K, N)=\sum_{n=K+1}^{K+N} \e(u_1n+\ldots+u_d n^d), \quad \vu = (u_1, \ldots, u_d) \in \Tor.
$$
We now give an upper bound on the mean value of  the largest value of all such sums, that is, for 
$\sup_{K \in \Z} \left |S_d(\vu; K,N)  \right |$. No result of this type seems to be known prior to this work.  


\begin{cor} 
\label{cor:MVT-Short} 
For any real positive $\rho \le d^2 +2d -1$ we have 
$$
\int_{\Tor} \sup_{K\in \Z} \left| S_d(\vu; K, N)\right|^{\rho }d\vu \le N^{  \mu_d \rho  +o(1)},  \qquad N\to \infty, 
$$  
where  
$$
\mu_d = 1- \frac{d} {d^2 +2d -1}.
$$ 
\end{cor}

Recalling the definition of~\eqref{eq:Tphi},~\eqref{eq:split}, for $\vx\in \T_k$ we can write 
$$
\sup_{\vy\in T_{d-k}} |T_{\va, \bphi}(\vx, \vy; N)|=\sup_{\vu\in T_d\, \cap \,\pi_{d, k}^{-1}(\vx) } |T_{\va, \bphi}(\vu; N)|,
$$
where $\pi_{d, k}$ is the orthogonal projection of $\Tor$ onto $\T_k$, that is, 
$$ 
\pi_{d, k}:~(u_1, \ldots, u_d) \to (u_1, \ldots, u_k).
$$
This suggests that one can try to extend Theorem~\ref{thm:general} to more general setting by taking  some other mappings instead of the orthogonal projection $\pi_{d, k}$.

Now, more generally, given a mapping $f: \R^d\rightarrow \R^k$, for $\vx\in \R^k$  we  denote 
$$
f^{-1}(\vx)=\{\vu\in \R^d:~f(\vu)=\vx\}.
$$
Then for $\rho>0$ we define
$$
\sM_{k, f,\rho} (\va, \bphi, N)=\int_{\R^k} \sup_{\vu\in \T_d\, \cap\, f^{-1}(\vx)}|T_{\va, \bphi}(\vu; N)|^\rho d \vx,
$$
where $T_{\va, \bphi}(\vu; N)$ is given by~\eqref{eq:Tphi}.
%

In the following we  first take $f$ to be an orthogonal projection onto some $k$-dimensional subspace, and second we take $f$ to be  some  H\"older mapping.  

Let $\cG(d, k)$ denote the collections of all $k$-dimensional  linear subspaces of $\R^d$. For $\cV\in \cG(d, k)$,  let $\pi_\cV:~\R^{d} \to \cV$ denote the orthogonal projection onto $\cV$.

For the degree sequence  $\deg \varphi_1, \ldots, \deg \varphi_d$ we denote them as 
$$ 
r_1\le \ldots\le r_d,
$$
and define 
\begin{equation}
\label{eq:wsigma}
\widetilde{ \sigma}_k(\bphi)=\sum_{i=k+1}^d r_i.
\end{equation}

\begin{theorem} 
\label{thm:V-projection} 
Suppose that $\bphi \in \Z[T]^d$ is such that 
  the Wronskian  $W(T;\bphi)$ 
does not vanish identically. Let   $\vec{a} = (a_n)_{n=1}^\infty$ 
be a sequence of complex  weights  with $a_n=n^{o(1)}$.  If $\cV\in \cG(d, k)$, 
then for any real positive $\rho \le  2 s(d) +d -k$ we have 
$$
\sM_{k, \pi_\cV,\rho} (\va, \bphi, N) \le N^{\mu_{\cV}(\bphi,k) \rho  +o(1)},  \qquad N\to \infty, 
$$ 
where  
$$
\mu_{\cV} (\bphi,k) = \frac{s(d)+\widetilde{ \sigma}_k(\bphi)+d-k } {2 s(d) +d -k}.
$$ 
\end{theorem}


Now we turn to {\it $\vartheta$-H\"older functions\/}, for some $0<\vartheta\le 1$, that is, functions $f:\R^d\rightarrow \R^k$ which satisfy 
$$
\|f(\vx)-f(\vy)\|\ll \|\vx-\vy\|^\vartheta,
$$
where $\| \vec{z}\|$ is the Euclidean norm of $\vec{z}$ (note that the left side of this inequality is the Euclidean norm in $\R^k$, while the right side is the Euclidean norm in $\R^d$). 
In particular, in the case $\vartheta=1$ the function $f$ is often called a {\it Lipschitz function\/}.

\begin{theorem} 
\label{thm:f-projection} 
Let  $f:\R^d\rightarrow\R^k$ be a $\vartheta$-H\"older map for some $0<\vartheta\le 1$. Let   $\vec{a} = (a_n)_{n=1}^\infty$ 
be a sequence of complex  weights  with $a_n=n^{o(1)}$.
Suppose that $\bphi \in \Z[T]^d$ is such that the Wronskian  $W(T;\bphi)$ 
does not vanish identically. Denote 
$$\delta(\bphi)  =\min_{i=1, \ldots, d} \deg \varphi_i.$$ 
Then for any real positive $\rho \le  2 s(d) +d -k\vartheta$ we have 
$$
\sM_{k,f,\rho} (\va, \bphi, N) \le N^{  \mu_{\vartheta}(\bphi,k) \rho  +o(1)},  \qquad N\to \infty, 
$$ 
where  
$$
\mu_{\vartheta} (\bphi,k) = \frac{s(d)+\sigma_0(\bphi)+d-(\delta(\bphi)+1) \vartheta k } {2 s(d) +d -k\vartheta}.
$$ 
\end{theorem}

Note that for the classical choice of polynomials~\eqref{eq:classical}, we have $\delta(\bphi)=1$, $\sigma_0(\bphi)=s(d)$ and thus 
$$
\mu_{\vartheta} (\bphi,k) = \frac{2s(d)+d-2  k \vartheta } {2 s(d) +d -k\vartheta}<1
$$
for any $k> 0$.  
Therefore  Theorem~\ref{thm:f-projection} gives a non-trivial bound for $\sM_{k,f,\rho}(\va, \bphi, N)$ for any $k> 0$ and any  $\vartheta$-H\"older function $f$ with $0<\vartheta\le 1$.  
Moreover if $\vartheta\rightarrow 0$ then  $\mu_{\vartheta} (\bphi,k) \rightarrow 1$. Indeed, it is  expected that if the function  $f$ becomes  ``bad" ($\vartheta$ becomes ``small'') then the bounds for $\sM_{k,f,\rho}(\va, \bphi, N)$ also become ``bad".  However, we do not know whether  there exits a continuous  function $f:\R^d\rightarrow \R^k$ such that we do not have  non-trivial bounds for $\sM_{k,f,\rho}(\va, \bphi, N)$, that is, for any 
$\rho>0$ and any $\varepsilon>0$, one has 
$$
\sM_{k,f,\rho}(\va, \bphi, N)\gg N^{1-\varepsilon}
$$
for infinitely many  $N\in \N$.

\subsection{Hybrid mean value theorems for discrepancy}
Similar to  works of Wooley~\cite{Wool3} and the authors~\cite{ChSh-New}, we obtain similar results for the discrepancy. 

Let $\xi_n$, $n\in \N$,  be a sequence in $[0,1)$. The {\it discrepancy\/}  of this sequence at length $N$ is defined as 
\begin{equation}
\label{eq:Discr}
D_N = \sup_{0\le a<b\le 1} \left |  \#\{1\le n\le N:~\xi_n\in (a, b)\} -(b-a) N \right |.
\end{equation} 
We note that sometimes in the literature the scaled quantity $N^{-1}D_N $ is   called the discrepancy, but 
since our argument looks cleaner with the definition~\eqref{eq:Discr}, we adopt it here. 

For $\vx\in \T_k$, $\vy\in \T_{d-k}$ we consider the  sequence 
$$
\sum_{j=1}^{k} x_j \varphi_j(n) +\sum_{j=1}^{d-k} y_j \varphi_{k+j}(n), \qquad n\in \N,
$$ 
and for each $N$ we denote by  $D_{\bphi}(\vx, \vy; N)$ the corresponding discrepancy of its fractional parts.

For $\rho>0$ let
$$
\fM_{k,\rho} (\bphi, N)=\int_{\T_k} \sup_{\vy\in \T_{d-k}} D_{\bphi}(\vx, \vy; N)^\rho d \vx.
$$

\begin{theorem}
\label{thm:discrepancy}
Suppose that $\bphi \in \Z[T]^d$ is such that the Wronskian  $W(T;\bphi)$ does not vanish identically.  Then for any  $1\le \rho \le  2 s(d) +d -k$ we have 
$$
\fM_{k,\rho} (\bphi, N) \le N^{ \mu(\bphi,k)\rho  +o(1)},  \qquad N\to \infty, 
$$ 
where  $\mu(\bphi,k)$ is as  Theorem~\ref{thm:general}.
\end{theorem}

From Theorem~\ref{thm:discrepancy} we derive a bound on the mean value of discrepancy over short intervals.  More precisely, for each $K\in \Z$ denote by $D_d(\vu; K, N)$  the discrepancy  of the sequence  of fractional parts
$$
\{u_1n+\ldots+u_d n^d\}, \qquad  n=K+1, \ldots, K+N.
$$

\begin{cor} 
\label{cor:MVT-Short-D} 
For any real positive $1\le \rho \le d^2 +2d -1$ we have 
$$
\int_{\Tor} \sup_{K\in \Z} D_d(\vu; K,  N)^{\rho }d\vu \le N^{  \mu_d\rho  +o(1)},  \qquad N\to \infty, 
$$ 
where  $\mu_d$ is given by Corollary~\ref{cor:MVT-Short}.
\end{cor}

\section{Preliminaries}

\subsection{Packing  large trigonometric polynomials in boxes} 
We first need a box-counting estimate from ~\cite[Lemma~3.7]{ChSh-New}. 


\begin{lemma} 
\label{lem:counting}
Let $0<\alpha<1$ and let  $\eps$  be  sufficiently small. For each $j=1, \ldots, d$ let 
\begin{equation}
\label{eq:zetaj}
\zeta_j=1/ \rf{N^{e_j+1+\eps-\alpha}},
\end{equation}
where  
$
e_j=\deg(\varphi_j)$, $j=1, \ldots, d.
$
We divide $\Tor$ into 
$$
U = \(\prod_{j=1}^d \zeta_j\) ^{-1}
$$ 
boxes of the form
$$
[n_1\zeta_1, (n_1+1)\zeta_1)\times \ldots \times [n_d\zeta_d, (n_d+1)\zeta_d),
$$
where $n_j=1, \ldots, 1/\zeta_j$ for each $j=1, \ldots, d$. Let $\fR$ be the collection of these boxes, and 
$$
\widetilde \fR=\{\cR \in \fR:~\exists\, \vu\in \cR \text{ with } |T_{\va, \bphi}( \vu; N)|\ge N^{\alpha}\}.
$$
Then, uniformly over $\alpha$,  we have 
$$
\# \widetilde \fR
 \le U N^{s(d)(1-2\alpha)+o(1)}.
$$
\end{lemma}

\subsection{The measure of the set of large Weyl sums}

We need a  slightly more general version of~\cite[Corollary~3.8]{ChSh-New}. 
We recall that the result of~\cite[Corollary~3.8]{ChSh-New} is formulated with an $T = N^\alpha$ for a fixed
real $\alpha$, however examining the argument one can easily see that as Lemma~\ref{lem:counting},  it is uniform with respect to  $\alpha$
and thus works for an arbitrary  parameter $T\ge 1$.   
More precisely, we have:

\begin{lemma} 
\label{lem:Markov}
Let $1 \le T \le N$.  Then 
\begin{align*}
\lambda  (\{ \vx\in \T_k:~ \exists\, \vy \in \T_{d-k}  &\text{ with } |T_{\va, \bphi}(\vx, \vy; N)|\ge T \}   ) \\
 & \le N^{s(d)+\sigma_k(\bphi)+d-k +o(1)} T^{-2  s(d) - d +k }. 
\end{align*}
\end{lemma}

The following  result is  similar to the result of Lemma~\ref{lem:Markov},  with the change of $\widetilde{ \sigma}_k(\bphi)$ only. 
Let $\cV$ be a $k$-dimensional subspace of $\R^d$. Also recall that  $\widetilde{ \sigma}_k(\bphi)$ is given by~\eqref{eq:wsigma}. For the orthogonal projection map $\pi_{\cV}$, from~\cite[Corollary~3.11]{ChSh-New} we have the following.   
\begin{lemma}
\label{lem:Markov-V}
Let $1 \le T \le N$.  Then 
\begin{align*}
\lambda  (\{ \vx\in \cV:~ \sup_{\vu\in \T_d \,\cap \,\pi_{\cV}^{-1}(\vx)} & |T_{\va, \bphi}(\vu; N)|\ge T \}   ) \\
 & \le N^{s(d)+\widetilde{ \sigma}_k(\bphi)+d-k +o(1)} T^{-2  s(d) - d +k }. 
\end{align*}
\end{lemma}

We now turn to $\vartheta$-H\"older functions $f:R^d\rightarrow \R^k$. Indeed,  applying the similar methods to the proofs
 of~\cite[Corollaries~3.8 and~3.11]{ChSh-New} we obtain Lemma~\ref{lem:Markov-Holder} below. Since  $\vartheta$-H\"older functions 
 do not appear in~\cite{ChSh-New},  for the sake of completeness, we  outline  the proof. 

\begin{lemma}
\label{lem:Markov-Holder}
Let $f:\R^d\rightarrow \R^k$ be a $\vartheta$-H\"older function.      Then for $1 \le T \le N$ we have 
\begin{align*}
\lambda  (\{ \vx\in \R^k:~ \sup_{\vu\in \T_d \,\cap \,f^{-1}(\vx)} & |T_{\va, \bphi}(\vu; N)|\ge T \}   ) \\
 & \le N^{s(d)+\sigma_0(\bphi)+d-(\delta(\bphi) +1) \vartheta k+o(1)} T^{-2  s(d)-d +\vartheta k}. 
\end{align*}
\end{lemma}

\begin{proof}  First of all suppose that $T=N^{\alpha}$ for some $0<\alpha<1$. We fix some sufficiently small   $\eps>0$ and define  the set 
$$
\fU = \bigcup_{\cR\in \widetilde \fR} \cR. 
$$
For $\cA\subseteq \R^d$ denote $f(\cA)=\{f(\vx): \vx\in \cA\}$. Observe that 
\begin{equation}\label{eq:cover-holder}
\begin{split}
\{ \vx\in \R^k:~ \sup_{\vu\in \T_d \,\cap \,f^{-1}(\vx)}  |T_{\va, \bphi}&(\vu; N)|\ge T \} \\
& \subseteq f\( \fU \)\subseteq \bigcup_{\cR\in \widetilde \fR} f(\cR), 
\end{split}
\end{equation}
where $ \widetilde \fR$ is as in Lemma~\ref{lem:counting}.

Since $f$ is  $\vartheta$-H\"older, for any $\cA\subseteq \R^d$ we obtain 
$$
\lambda(f(\cA))\ll \( \diam\cA\)^{\vartheta k}, 
$$
where $\diam \cA=\sup\{\|\va-\vb\|:~\va, \vb \in \cA\}$.  For each $\cR\in \widetilde\fR$, by~\eqref{eq:zetaj} we have 
$$
\diam \cR\ll N^{\alpha-1-\delta(\bphi)}.
$$
Combining with Lemma~\ref{lem:counting} and the estimate~\eqref{eq:cover-holder}, we derive 
\begin{align*}
\lambda(\{ \vx\in \R^k:~ \sup_{\vu\in \T_d \,\cap \,f^{-1}(\vx)}  &|T_{\va, \bphi}(\vu; N)|\ge T \}) \\
& \le \# \widetilde \fR \(\diam \cR\)^{\vartheta k}\\
&\ll  U N^{s(d)(1-2\alpha)}N^{(\alpha-\delta(\bphi)-1-\varepsilon)\vartheta k}
\end{align*}
Since  $\eps$ is arbitrary, recalling the value of $U$, we now obtain 
\begin{align*}
\lambda(\{ \vx\in \R^k:~& \sup_{\vu\in \T_d \,\cap \,f^{-1}(\vx)}  |T_{\va, \bphi}(\vu; N)|\ge T \}) \\
&\quad \le N^{s(d)+\sigma_0(\bphi)+d-(\delta(\bphi) +1) \vartheta k+o(1)} N^{-2  s(d)\alpha-d \alpha+\vartheta k \alpha}.
\end{align*}
Recalling that $T=N^{\alpha}$, we now obtain the desired result by taking $T=N^{\alpha}$. 
\end{proof}

\subsection{Discrepancy and exponential sums}
We recall  the classical {\it Erd\H{o}s--Tur\'{a}n inequality\/} (see, for instance,~\cite[Theorem~1.21]{DrTi}).

\begin{lemma}
\label{lem:ET}
Let $\xi_n$, $n\in \N$,  be a sequence in $[0,1)$. Then for the discrepancy $D_N$ given by~\eqref{eq:Discr} and any $G\in \N$, we have  
$$
D_N \le 3 \left( \frac{N}{G+1} + \sum_{g=1}^{G}\frac{1}{g} \left| \sum_{n=1}^{N} \e(g \xi_n) \right | \right).
$$
\end{lemma}

\section{Proofs of Mean Value Theorems for Exponential Sums}

\subsection{Proof of Theorems~\ref{thm:general},~\ref{thm:V-projection},~\ref{thm:f-projection}}

Theorems~\ref{thm:general},~\ref{thm:V-projection} and~\ref{thm:f-projection} follows by combining
Lemmas~\ref{lem:Markov},~\ref{lem:Markov-V} and~\ref{lem:Markov-Holder}, respectively, 
with Lemma~\ref{lem:level-set}   below.

\begin{lemma}
\label{lem:level-set}
Let $N$ be a positive large number and $F:\T_k\rightarrow [0,N]$ be a function. Suppose that there exists positive constants $a<b$ such that for any $1\le T\le N$,
\begin{equation}
\label{eq:level-set}
\lambda(\{\vx\in \T_k:~F(\vx)\ge T\})\le N^aT^{-b}.
\end{equation}
Then for any positive $\rho\le b$,
$$
\int_{T_k} F(\vx)^{\rho} d\vx \ll  N^{\frac{\rho a}{b}}\log N.
$$
\end{lemma} 
\begin{proof}
Let $R=N^{a/b}$. Note that for $T>R$ we have a nontrivial estimate in~\eqref{eq:level-set}. We partition $ \T_k$ into sets $\cX_\ell$, $\ell\in \N$,  where
$$
\cX_0=\{\vx\in \T_k:~F(\vx)\le R\},
$$ 
and for $\ell\ge 1$, 
$$
\cX_\ell=\{\vx\in \T_k:~2^{\ell-1} R< F(\vx)\le 2^{\ell} R\}.
$$ 
By our assumption~\eqref{eq:level-set} we have 
$$
\lambda(\cX_{\ell})\ll N^{a}(2^{\ell} R)^{-b}.
$$
Clearly for some $L=O(\log N)$ we have $\cX_{\ell}=\emptyset$. Therefore, 
\begin{align*}
\int_{\T_k} F(\vx)^{\rho} d\vx &=\sum_{\ell=0}^{L} \int_{\cX_{\ell}} F(\vx)^{\rho} d\vx\\
&\ll R^{\rho} +\sum_{\ell=1}^{L} (2^{\ell }R)^{\rho} N^{a}(2^{\ell} R)^{-b}\\
&\ll R^\rho +N^a R^{\rho-b}\sum_{\ell=1}^{L}2^{\ell(\rho-b)}.
\end{align*}
By the choice of $R=N^{a/b}$ and the condition that $\rho\le b$ we obtain the desired bound.
\end{proof}

\subsection{Proof of Corollary~\ref{cor:MVT-Short}}

For $K\in \Z$, recall that  Weyl sums over short intervals  are defined as follows 
$$
S_d(\vu; K,N) = \sum_{n=K+1}^{K+N} \e(u_1n+\ldots+u_d n^d). 
$$
 We write 
$$
S_d(\vu; K,N)=\sum_{n=1}^N \e(u_1(n+K)+\ldots+u_d(n+K)^d).
$$
We observe that in the  polynomial identity  
\begin{align*}
u_1(T+K)+\ldots & +u_d (T+K)^d \\
& = v_0+v_1T+\ldots+v_{d-1}T^{d-1}+u_d T^d \in \R[T], 
\end{align*}
where for $j=0, 1, \ldots, d-1$, each $v_j$, depends only on $u_1, \ldots, u_{d}$ and $K$.  It follows that 
\begin{equation}
\label{eq:shift}
\sup_{K\in \Z} |S_d(\vu; K, N)|\le \sup_{(v_1, \ldots, v_{d-1})\in \T_{d-1}} |S_d((v_1, \ldots, v_{d-1}, u_d); N)|.
\end{equation}
Note that for any fixed $u_d$ for any   $(u_1, \ldots, u_{d-1})\in T_{d-1}$ the estimate~\ref{eq:shift} holds for $\vu=(u_1, \ldots, u_d)$. Thus we obtain 
\begin{align*}
\int_{T_d} & \sup_{K\in \Z} |S_d(\vu; K, N)|^\rho d\vu\\
&\qquad =\int_{0}^1\left (\int_{T_{d-1}} \sup_{K\in \Z}|S_d(\vu; K, N)|^{\rho}du_1\ldots du_{d-1}\right) du_d\\
&\qquad \le \int_{0}^1\sup_{(v_1, \ldots, v_{d-1})\in \T_{d-1}} |S_d((v_1, \ldots, v_{d-1}, u_d); N)|^\rho d u_d. 
\end{align*} 
Hence Theorem~\ref{thm:general}, applied $k=1$, $\varphi_1(T)=T^d$,   $\varphi_i(T)=T^{i-1}$ for $i=2, \ldots, d$, 
and thus with $\sigma_{1} (\bphi) = d(d-1)/2$,  
yields the desired bound. 


\section{Proofs of Mean Value Theorems for the Discrepancy} 

\subsection{Proof of Theorem~\ref{thm:discrepancy}}

For any $\vx\in \T_k, \vy\in \T_{d-k}$, by Lemma~\ref{lem:ET} for  any $G\in \N$ we obtain 
$$
D_{\bphi}(\vx, \vy; N)\ll \frac{N}{G}+\sum_{g=1}^{G}\frac{1}{g} \left |T_{\bphi}(g\vx, g\vy; N)\right |,
$$
and therefore 
\begin{equation}
\label{eq:ET}
\sup_{\vy\in \T_{d-k}}D_{\bphi}(\vx, \vy; N)\ll \frac{N}{G}+\sum_{g=1}^{G}\frac{1}{g}  \sup_{\vy\in \T_{d-k}}\left |T_{\bphi}(g\vx, \vy; N)\right |.
\end{equation}

We now use the following  invariant property of Lebesgue measure on a torus.

\begin{lemma}
\label{lem:invairant}
Let $F:\T_k\rightarrow [0, N]$ be a continuous   function. Then for any  integer $g\ne 0$ we have 
\begin{equation}
\label{eq:invairiant}
\int_{\T_k} F(g\vx) d\vx=\int_{\T_k} F(\vx) d\vx.
\end{equation}
\end{lemma}

\begin{proof} 
For any Borel set $\cA\subseteq \T_k$ and any  integer  $g\neq 0$,  we have  (for a proof see~\cite[Section~3]{Wool3})
$$
\lambda(\{\vx\in \T_k:~g \vx \in \cA\})=\lambda(\cA),
$$
which is the same as the identity
$$
\int_{\T_k} {\bf 1}_{\cA}(g\vx) d\vx=\int_{\T_k} {\bf 1}_{\cA}(\vx) d\vx, 
$$
where ${\bf 1}_{\cA}$ is the characteristic function of $\cA$. 
Thus~\eqref{eq:invairiant} holds when $F={\bf 1}_{\cA}$. It follows that the identity~\eqref{eq:invairiant} still holds when $F$ is a finite linear combination of characteristic functions, that is 
$$
F(\vx)=\sum_{j=1}^J a_j {\bf 1}_{\cA_j}(\vx), 
$$
for sets  $\cA_j\subseteq \T_k$, $j =1, \ldots, J$. 
Since any continuous function can be arbitrary  approximated by a  finite  linear combination of  such functions,  the desired identity follows.
\end{proof}

Here $\| f(\vx) \|_\rho$ denote the $L^\rho(\T_k)$-norm of a function $f$ on $\T_k$. 
Then by~\eqref{eq:ET} and  the Minkowski inequality,
\begin{equation}
\label{eq:Discr-rho}
\left \|\sup_{\vy\in \T_{d-k}} D_{\bphi}(\vx, \vy; N) \right \|_{\rho}\ll \frac{N}{G}+\sum_{g=1}^{G}\frac{1}{g} \left \| \sup_{\vy\in \T_{d-k}}\left |T_{\bphi}(g\vx, \vy; N)\right |\right \|_{\rho}. 
\end{equation}

For any positive integer $g\ne 0$,  Lemma~\ref{lem:invairant} implies 
$$
\left \| \sup_{\vy\in \T_{d-k}}\left |T_{\bphi}(g\vx, \vy; N)\right |\right \|_{\rho}
= \left \| \sup_{\vy\in \T_{d-k}}\left |T_{\bphi}(\vx, \vy; N)\right |\right \|_{\rho}, 
$$
which together with~\eqref{eq:Discr-rho}    and Theorem~\ref{thm:general} yileds
\begin{align*}
\left \|\sup_{\vy\in \T_{d-k}} D_{\bphi}(\vx, \vy; N) \right \|_{\rho} &\ll \frac{N}{G}+ \left \| \sup_{\vy\in \T_{d-k}}\left |T_{\bphi}(\vx, \vy; N)\right |\right \|_{\rho} \log G \\
&\ll \frac{N}{G}+ N^{\mu(\bphi, k)+o(1)}\log G. 
\end{align*}
Choosing $G=N^{1-\mu(\bphi, k)}$, we derive the desired  result.

\subsection{Proof of Corollary~\ref{cor:MVT-Short-D}}

Recall that  $D_d(\vu;K, N)$  is  the discrepancy  of the sequence  of fractional parts
$$
\{u_1n+\ldots+u_d n^d\}, \qquad  n=K+1, \ldots, K+N.
$$
Clearly this sequence is same as  
$$
\{u_1(n+K)+\ldots+u_d (n+K)^d\}, \qquad  n=1, \ldots, N, 
$$
and thus as  before, see~\eqref{eq:shift}, we see that this sequence is the same as   
$$
\{v_0+v_1n+\ldots+v_{d-1}n^{d-1}+u_d n^d\}, \qquad  n=1, \ldots, N, 
$$
where for $j=0, 1, \ldots, d-1$, each $v_j$, depends only on $u_1, \ldots, u_{d}$ and $K$.  Furthermore let $\vu^*=(v_1, \ldots, v_{d-1}, u_d)$. 
It is not hard to see that the influence of the discarded constant term $v_0$ can be absorbed in a constant factor and does not 
change the order of magnitude of the discrepancy. More, precisely,  we have   
$$
D_d(\vu^*;  N)\ll D_d(\vu; K, N)\ll D_d(\vu^*; N),
$$
where the implied constant is  absolute.  It follows that 
$$
\sup_{K\in \Z} D_d(\vu; K, N)\ll \sup_{(v_1, \ldots, v_{d-1})\in \T_{d-1}} D_d((v_1, \ldots, v_{d-1}, u_d), N). 
$$

Using the similar arguments as in the proof of Corollary~\ref{cor:MVT-Short} and applying  Theorem~\ref{thm:discrepancy}, with  $k=1$, $\varphi_1(T)=T^d$,  $\varphi_i(T)=T^{i-1}$ for $i=2, \ldots, d$, 
and thus with $\sigma_{1} (\bphi) = d(d-1)/2$,  
we obtain the desired bound.

%
%
%
%

\section{Comments and Open Questions}

\subsection{Special cases of $d=2$ and $d=3$}  
We show that for  special  cases of $d=2$ and $d=3$ for the moments
\begin{equation}
\begin{split}
\label{eq:special}
& \cM_{2,2} (N) =  \int_{0}^{1} \sup_{y\in [0,1]} \left|\sum_{n=1}^{N}\e(xn^2+yn)\right|^2 dx, \\
&  \cM_{3,4} (N) = \int_{0}^{1} \sup_{y,z\in [0,1]} \left|\sum_{n=1}^{N}\e(xn^3+yn^2 + zn)\right|^4 dx,
\end{split} 
\end{equation} 
 we have  better bounds (nearly optimal) than the bounds in Theorem~\ref{thm:general}.  
 
Applying Theorem~\ref{thm:general} with $d=2$, $k=1$, $ \sigma_k(\bphi)=1$,  $\rho=2$ 
and  with $d=3$, $k=1$, $ \sigma_k(\bphi)=3$,  $\rho=4$  
we derive that 
$$
 \cM_{2,2} (N) \le N^{10/7+o(1)} \mand  \cM_{3,4} (N) \le  N^{22/7+o(1)} ,
$$
respectively. 
On the other hand,  we have the lower bounds
$$
 \cM_{2,2} (N) \ge \int_{0}^{1} \int_{0}^{1} \left|\sum_{n=1}^{N}\e(xn^2+yn)\right|^2 dy dx = N, 
$$
and 
\begin{align*}
\cM_{3,4} (N) & \ge \int_{0}^{1} \int_{0}^{1}  \int_{0}^{1}  \left|\sum_{n=1}^{N}\e(xn^3+yn^2 + zn)\right|^4  dz dy dx\\
& \ge 2N^2 + O(N),
\end{align*}
see, for example,~\cite[Section~1]{Wool2}. 

Now we use a different way to bound the square mean values~\eqref{eq:special} and  reduce
 the gap between the above lower and upper bounds. 

For $d=2$ we need  the following  well known inequality, for a proof see~\cite[Inequality~(8.8)]{IwKow}:
$$
\sup_{y\in [0,1]} \left|\sum_{n=1}^{N}\e(xn^2+yn)\right|^2\ll N+\sum_{h=1}^{N}\min\left \{\frac{1}{\|2hx\|}, N\right \}, 
$$
where   $\|x\|=\min\{|x-n|:~n\in \Z\}$  denotes  the distance of $x$ to the nearest integer. 
Moreover, for any positive integer $h$, applying  Lemma~\ref{lem:invairant} 
we obtain that 
uniformly over $h$, we have 
$$
\int_{0}^{1}  \min\left \{\frac{1}{\|2hx\|}, N\right \} dx = \int_{0}^{1}  \min\left \{\frac{1}{\|x\|}, N\right \} dx \ll \log N,
$$
hence 
$$
 \cM_{2,2} (N)\ll N\log N.
$$
Note that this is tight except, possibly, for  the logarithm factor.

For $d=3$ we use that by~\cite[Proposition~8.2]{IwKow}:
\begin{align*}
\sup_{y,z\in [0,1]} \left|\sum_{n=1}^{N}\e(xn^3+yn^2 + zn)\right|^4&\ll 
N \sum_{g=-N}^{N}\sum_{h=-N}^{N}\min\left \{\frac{1}{\|6ghx\|}, N\right \}\\
& \ll N^3 +
N \sum_{\substack{g,h=-N\\gh\ne 0}}^{N}\min\left \{\frac{1}{\|6ghx\|}, N\right \},
\end{align*}
and arguing as before we obtain 
$$
 \cM_{3,4} (N)\ll N^3\log N.
$$

It is natural to ask whether similar arguments can improve Theorem~\ref{thm:general} for higher degrees. 
Note that  typically methods based on the Vinogradov mean value  theorem yield better bounds for higher degrees. 
Indeed this also happens here, which means  that Theorem~\ref{thm:general} gives  stronger 
bounds than using~\cite[Proposition~8.2]{IwKow} for $d\ge 4$ and applying the above arguments. 

\subsection{Open questions}
It is interesting to try to use the ``self-improving'' idea of~\cite{ChSh-New} and in particular~\cite[Corollary~3.9]{ChSh-New}
to obtain stronger results in the case when one of the polynomials $\varphi_{k+1} \ldots, \varphi_d$
is linear. There are however some obstacles which the authors have not been able to overcome.

It is also natural to study the mixed mean values
$$
M_{k, \rho, \tau } (\va, \bphi, N) = \int_{\vx \in [0,1]^k}  \(\int_{\vy \in [0,1]^{d-k}}
\left|T_{\va, \bphi}( \vx, \vy; N)  \right|^\tau d \vy\)^\rho d\vx, 
$$
and obtain   bounds which are stronger than those following from the trivial inequality
$$
M_{k, \rho, \tau } (\va, \bphi, N) \le M_{k, \rho\tau } (\va, \bphi, N) 
$$
and Theorem~\ref{thm:general}. The case of the classical Weyl sums is of special interest.

\section*{Acknowledgement}

This work was  supported   by ARC Grant~DP170100786.

\end{document}